\documentclass[12pt,oneside]{amsart}
\usepackage[latin1]{inputenc}
\usepackage[T1]{fontenc}
\usepackage{lmodern}
\usepackage{amsmath,amssymb,amsfonts,a4}
\usepackage{hyperref}
\usepackage{textcmds}
\usepackage[all]{xy}
\entrymodifiers={+!!<0pt,\the\fontdimen22\textfont2>}
\SelectTips{cm}{10}
\theoremstyle{plain}
\newtheorem{thm}{Theorem}
\newtheorem{lem}[thm]{Lemma}
\newtheorem{constthm}[thm]{Construction--Theorem}
\newtheorem{cor}[thm]{Corollary}
\newtheorem{prop}[thm]{Proposition}
\theoremstyle{definition}
\newtheorem{defn}[thm]{Definition}
\newenvironment{examples}[1][]{\refstepcounter{thm}\noindent{\textbf{Examples \thethm{}. #1}}}{}
\newtheorem{rmk}[thm]{Remark}
\newtheorem{rmks}[thm]{Remarks}
\numberwithin{thm}{section}
\numberwithin{equation}{section}
\renewcommand{\emptyset}{\varnothing}
\renewcommand{\epsilon}{\varepsilon}
\newcommand{\otimeshat}{{\widehat{\otimes}}}
\newcommand{\divise}{\mkern2mu|\mkern2mu}
\newcommand{\tors}[2]{{\vphantom{#2}}_{#1}{#2}}
\newcommand{\aKa}{$\tors{a}{\mathcal{K}}_a$}
\newcommand{\pitop}{\pi_1^{{\rm top}}}
\newcommand{\pitildetop}{\pitilde^{{\rm top}}}
\newcommand{\Xtildetop}{\Xtilde^{{\rm top}}}
\newcommand{\Br}{{\mathrm{Br}}}
\newcommand{\Gm}{\mathbf{G}_\mathrm{m}}
\newcommand{\Xhat}{{\hat X}}

\newcommand{\Zhat}{{\hat \Z}}
\newcommand{\Vhat}{{\hat V}}
\newcommand{\kpi}{$K(\pi,1)$}
\newcommand{\kbar}{{\bar k}}
\newcommand{\Xbar}{{\bar X}}

\newcommand{\Ybar}{{\bar Y}}
\newcommand{\Zbar}{{\bar Z}}
\newcommand{\Xtilde}{\tilde{X}}
\newcommand{\atilde}{\tilde{a}}
\newcommand{\alphatilde}{\tilde{\alpha}}
\newcommand{\pitilde}{\tilde{\pi}}
\newcommand{\mmu}{{\boldsymbol{\mu}}}
\newcommand{\Pic}{{\rm Pic}}
\newcommand{\Et}{{\sf Et}}
\newcommand{\Sets}{{\sf Sets}}
\newcommand{\Spec}{{\rm Spec}}
\newcommand{\Gal}{{\rm Gal}}
\newcommand{\sB}{{\mathcal B}}
\newcommand{\sC}{{\mathcal C}}
\newcommand{\sP}{{\mathcal P}}
\newcommand{\R}{{\mathbf R}}
\newcommand{\C}{{\mathbf C}}

\renewcommand{\H}{H}
\newcommand{\Q}{{\mathbf Q}}
\newcommand{\Z}{{\mathbf Z}}
\newcommand{\Zl}{{\Z_{\ell}}}
\newcommand{\Ql}{{\Q_{\ell}}}
\newcommand{\sX}{{\mathcal X}}
\def\tilde{\widetilde}
\title[]{Remarks on cycle classes of sections of the arithmetic fundamental group}
\author{H\'el\`ene Esnault}
\address{Universit\"at Duisburg--Essen, Mathematik, 45117 Essen, Germany}
\email{esnault@uni-due.de}
\author{Olivier Wittenberg}
\address{Institut de Recherche Math\'ematique Avanc\'ee, CNRS \ndash{} Universit\'e Louis Pasteur, 7 rue Ren\'e Descartes, 67084 Strasbourg Cedex, France}
\curraddr{D\'epartement de math\'ematiques et applications, \'Ecole normale sup\'erieure, 45~rue d'Ulm, 75320 Paris Cedex 05, France}
\email{wittenberg@dma.ens.fr}
\thanks{Partially supported by the DFG Leibniz Preis and the SFB/TR 45}
\date{December 18, 2008}
\dedicatory{\`A Pierre Deligne, avec reconnaissance et admiration.}

\begin{document}
\parindent0cm
\parskip5pt
\begin{abstract}
Given a smooth and separated~\kpi{} variety~$X$ over a field~$k$, we associate a ``cycle class'' in
\'etale cohomology with compact supports to any continuous section of the natural map from the arithmetic fundamental group of~$X$ to the absolute Galois group of~$k$.
We discuss the algebraicity of this class in the case of curves over $p$\nobreakdash-adic fields.
Finally, an \'etale adaptation  of Beilinson's geometrization of the pronilpotent completion of the topological fundamental group allows us to lift this cycle class in suitable cohomology groups. 
\end{abstract}
\maketitle

\section{Introduction}
\label{secintro}
Let~$X$ be a geometrically connected variety over a field~$k$ of characteristic~$0$.  Denote by~$\kbar$ an algebraic closure of~$k$. Let $\Xbar = X \otimes_k \kbar$
and $G_k=\Gal(\kbar/k)$.

Grothendieck considered in~\cite{sga1} the category $\Et(X)$ of \'etale covers of~$X$ (\emph{i.e.},~of finite \'etale morphisms of schemes $\pi \colon Y \rightarrow X$)
and used it to define the fundamental group $\pi_1(X,x)$ of~$X$ based at a geometric point~$x$ of~$X$ as the automorphism group
of the fiber functor $\Et(X) \rightarrow \Sets$, $\pi \mapsto \pi^{-1}(x)$.  The structure morphism $\epsilon \colon X \rightarrow \Spec(k)$ induces
a map $\epsilon_\star\colon\pi_1(X,x) \rightarrow G_k$, since $\pi_1(\Spec(k),x)$ is canonically isomorphic to~$G_k$.  Grothendieck proved that~$\epsilon_\star$ is onto and that its kernel
identifies with $\pi_1(\Xbar,x)$. Moreover, for any other geometric point~$x'$ of~$X$, there exists an isomorphism $\pi_1(X,x)\simeq \pi_1(X,x')$ which is compatible
with the projections to~$G_k$.
Any rational point $a \in X(k)$ therefore induces a section $a_\star \colon G_k \rightarrow \pi_1(X,x)$ of~$\epsilon_\star$, well-defined up to conjugacy by $\pi_1(\Xbar,x)$
(indeed, $a$ induces a canonical section of $\pi_1(X,a)\rightarrow G_k$).

Let~$G_{k(X)}$ denote the absolute Galois group of~$k(X)$.
The inclusion of the generic point of~$X$ induces a map $G_{k(X)} \rightarrow \pi_1(X,x)$
which is compatible with the projections to~$G_k$ and which is well-defined up to conjugacy by $\pi_1(\Xbar,x)$.
In \cite[Section~15]{DeP}, Deligne proves that for $a \in X(k)$, the section~$a_\star$ admits liftings to~$G_{k(X)}$, as pictured below:
$$
\xymatrix{
  G_{k(X)} \ar[d]\\
  \pi_1(X,x) \ar[r]^(.62){\epsilon_*}  &  \ar@<1ex>[l]^(.37){a_\star} G_k \ar@{.>}[lu]
}
$$
In terms of Galois groups, this can be seen as follows.  Assume for simplicity that~$X$ is a curve
(Deligne writes ``Par lassitude, nous ne traiterons que du cas o\`u $X$ est de dimension~$1$'').
The choice of a local parameter~$t$ of~$X$ at~$a$ determines an isomorphism $k(X)_a \simeq k((t))$, where $k(X)_a$ denotes the completion of~$k(X)$ at~$a$;
hence an embedding $k(X) \subset K$, where $K=\bigcup_{n\geq 1}k((t^{1/n}))$. The resulting map between absolute Galois
groups $G_K \rightarrow G_{k(X)}$ then provides a lifting of~$a_\star$. Indeed, the map $G_K \rightarrow G_k$ induced by the
inclusion $k \subset K$ is an isomorphism. 

The construction just described is easily seen to depend on the local parameter~$t$ only to order~$1$,
\emph{i.e.},~it only depends on the choice of a tangent vector~$\tau$ to~$X$ at~$a$.
Thus Deligne's {\it tangential base points} $(a, \tau)$ induce splittings of $G_{k(X)} \rightarrow G_k$.
Equivalently, for any dense open $U\subseteq X$, one can say that the datum of a nonzero tangent vector~$\tau$ on~$X$ at~$a$
enables one to define the fundamental group $\pi_1(U,(a,\tau))$ of~$U$ based at~$(a,\tau)$
as the automorphism group of the fiber functor
$\Et(U) \rightarrow \Sets$, $\pi \mapsto \pi^{-1}(\bar \eta)$ where $\bar \eta=\Spec\left(\bigcup_{n\geq 1}\kbar((t^{1/n}))\right)$.

Deligne's tangential base points~$(a,\tau)$ produce sections $G_k\to G_{k(X)}$ which factor through $G_{k(X)_a}$.
For a given $a \in X(k)$, the set of all sections (up to conjugacy) which satisfy this property is naturally a torsor under the group $H^1(k,\Zhat(1)):=\varprojlim_{n\geq 1} H^1(k, \mmu_n)$.
Those sections which in addition come from a tangential base point form a subtorsor under the image of $H^1(k,\Z(1)):=k^\times$ in $H^1(k,\Zhat(1))$;
one could say that they are motivic, as opposed to profinite.

In this note, given a smooth and geometrically connected separated variety~$X$ of dimension~$d$ over a field~$k$,
under the assumption that~$X$ is a~\kpi{} variety (see Definition~\ref{defn2.1} or~\cite[Appendix~A]{Stix})
we associate
to any section  $s \colon G_k \to \pi_1(X,x)$
a class in the \'etale cohomology group with compact supports $H^{2d}_c(X, \Z/N\Z(d))$  for any~$N$ invertible in~$k$
(see Theorem~\ref{thm2.5}).  We call it the cycle class of~$s$.
(Such a class had been considered by Mochizuki~\cite[Introduction, Structure of the Proof, (1)]{Mo}
in the case where~$X$ is proper and has dimension~$1$.  Theorem~\ref{thm2.5} below, on the other hand, may be used to associate cycle classes, on open varieties,
to sections~$s$ coming from tangential base points and even from ``profinite'' tangent vectors at infinity as alluded to in the previous paragraph.)
The construction we give resembles for \'etale cohomology with torsion coefficients Bloch's
decomposition of the diagonal for cohomology with nontorsion coefficients.
According to the latter,
 if the Chow group of $0$\nobreakdash-cycles satisfies $CH_0(X\otimes_k k(X)) \otimes_\Z {\Q}=\Q$, then, up to a cycle whose restriction to $X\otimes_k k(X)$ vanishes,
some nonzero multiple of the diagonal~$\Delta$ is rationally equivalent, on~$X \times X$, to $\alpha\times X$ for some $0$\nobreakdash-cycle~$\alpha$ on~$X$.
In Theorem~\ref{thm2.5} below, the \kpi{} assumption implies that 
 up to a class whose
pullback to $X\times X_s$ vanishes, where $X_s \rightarrow X$ denotes the $k$\nobreakdash-form of the universal cover of~$\Xbar$ determined by~$s$,
the cohomology class of~$\Delta$ is equal to $\alpha\times X$ for some $\alpha \in H^{2d}_c(X,\Z/N\Z(d))$.

If~$s$ is the section associated to a rational point $a \in X(k)$, then the cycle class of~$s$ in $H^{2d}_c(X,\Z/N\Z(d))$ coincides with the cycle class of~$a$.
As a consequence, in any situation in which one might hope that all sections come from rational points (for instance, Grothendieck's section conjecture predicts that this should be the case
if~$X$ is a proper curve of genus $g \geq 2$ and~$k$ is a finitely generated extension of~$\Q$), one first has to prove that the cycle class of any section~$s$ is actually the cycle class
of an algebraic cycle.  We consider this question in~\textsection\ref{sectionalgebraicity} under the assumption that~$X$ is a smooth proper curve over a $p$\nobreakdash-adic field~$k$.
With these hypotheses, Koenigsmann~\cite{Ko} established that if
$s \colon G_k \rightarrow \pi_1(X,x)$ is a section, then~$s$ comes from a rational point of~$X$ as soon as~$s$ lifts to a section of the natural map $G_{k(X)} \rightarrow G_k$
(this condition is necessary thanks to Deligne's construction).  The proof builds upon model-theoretic results due to Pop.
In Proposition~\ref{propkoenig} we show that in the situation of Koenigsmann's theorem, the liftings
of the cycle class $\alpha \in H^2(X,\Z/N\Z(1))$
of~$s$ to $H^2_c(U,\Z/N\Z(1))$ given by Theorem~\ref{thm2.5}
(for all dense open subsets $U \subseteq X$) 
 may be used to give a direct proof
of the weaker statement that~$\alpha$ is the cycle class of a divisor of degree~$1$ on~$X$ (well-defined up to linear equivalence).
We then investigate the algebraicity of the cycle class~$\alpha$ of~$s$ in the absence of any birational hypothesis.
Using Lichtenbaum's results about the period and the index of curves over $p$\nobreakdash-adic fields,
Stix~\cite{Stix2} was able to show that if~$X$ has genus~$g\ge 1$ and $\pi_1(X,x) \rightarrow G_k$ admits a section, then the index and the period of $X$ are powers of~$p$.
In the remainder of~\textsection\ref{sectionalgebraicity}, we prove that for any~$N$ which is prime to~$p$, the cycle class of any section $s \colon G_k \rightarrow \pi_1(X,x)$
is an algebraic cycle class.  From this we deduce in particular a new proof of Stix's theorem.

The last part of this note discusses a possible way to lift the cycle class
of a section $s \colon G_k\rightarrow \pi_1(X,x)$ to
cohomology groups which take into account the pronilpotent completion of
$\pi_1(\bar X,x)$. Over the field of complex numbers,
Beilinson~\cite[Section~3]{DelGon} constructed a cosimplicial
scheme~$\sP_a(X)$ with the property
that the Hopf algebra  $\varinjlim_{n\geq 1} {\rm Hom}_{\Q}(\Q[\pitop(X(\C), a)]/I^{n+1}, \Q)$
arises from the cohomology
of~$\sP_a(X)$,
where $\pitop(X(\C), a)$ denotes the topological fundamental group based at a point $a \in X(\C)$ and~$I$ denotes the augmentation ideal.
In~\textsection\ref{section4} we adapt part of Beilinson's description of $\Q[\pitop(X(\C),a)]/I^{n+1}$ to the \'etale fundamental group. The purpose is to replace the rational point $a$ by
an abstract section $s: G_k\to \pi_1(X,x)$ of Grothendieck's arithmetic fundamental group. If~$X$ has dimension~$1$, this section~$s$ does not have to come from
one of Deligne's tangential base points.
The only assumption we need to make is that~$X$ is a~\kpi{}, for example~$X$ could be a smooth proper curve of genus~$\ge 1$.
We~closely follow the topological description in~\cite[\textsection3.3, \textsection3.4]{DelGon}.
There, the authors express in cohomological terms not the $\Q$\nobreakdash-vector space $\Q[\pitop(X(\C),a)]/I^{n+1}$
but its dual. The latter turns out to coincide with the hypercohomology of the complex~\aKa{} described in {\it loc.\ cit.}
It is unclear how to define an analogous complex if one replaces~$a$ by an abstract section $s$. However, 
under the \kpi{} assumption, one may replace~$a$ by the $k$\nobreakdash-form of the universal cover of~$\Xbar$ defined by $s$ to obtain a complex, albeit on a larger space,
which gives the correct hypercohomology group (see Proposition~\ref{propcorrectcohomology}).

If one dualizes, that is, if one comes back to $\Q[\pitop(X(\C),a)]/I^{n+1}$
(see~\textsection\ref{sub3.1}), one finds a complex which is more difficult to
write down since it is defined only as an object of the derived category (see
Definition~\ref{defn3.5}). However, it turns out that the cohomology of this complex carries liftings
of the cycle class associated to $s \colon G_k \rightarrow \pi_1(X,x)$ (see Proposition~\ref{prop3.11}).
We discuss in more detail the additional information contained in the first of these liftings.

\medskip
{\it Acknowledgments.} We thank Alexander Beilinson for his interest, and Ph\`ung H\^o Hai and Marc Levine  for discussions based on earlier work related to the topic of this note. We are
indebted to Jakob Stix for a number of useful comments and for allowing us to include Remark~\ref{rmks37}~(ii).
It~is also a pleasure to acknowledge the influence of his period-index result~\cite{Stix2}
on the contents of~\textsection\ref{sectionalgebraicity}.

{\it Notation.} If~$M$ is an abelian group, $m$ is an integer and~$\ell$ is a prime number, we let $M \otimeshat \Zl=\varprojlim_{n \geq 1} M/\ell^nM$ and ${\rm T}_\ell(M)={\rm Hom}(\Ql/\Zl,M)$, and we denote by $M[m]$  the $m$\nobreakdash-torsion subgroup of~$M$.
Moreover we let $M \otimeshat \Ql = (M \otimeshat \Zl) \otimes_{\Zl}\Ql$ and ${\rm V}_\ell(M)={\rm T}_\ell(M) \otimes_{\Zl} \Ql$.

\section{Cycle classes of sections}
\label{section2}

Let~$k$ be a field and~$\kbar$ be a separable closure of~$k$.  We denote by~$X$ a geometrically connected separated variety over~$k$ endowed with a $\kbar$\nobreakdash-point $x \in X(\kbar)$,
and we let $\Xbar = X \otimes_k \kbar$.

Grothendieck~\cite{sga1} defines a \emph{universal cover of~$X$ at~$x$}
to be a directed inverse system of connected pointed \'etale covers of~$(X,x)$ such that the inverse limit
$\pitilde_x:\Xtilde_x\to X$ factors through every \'etale cover of~$X$.
(A~universal cover of~$X$ at~$x$ exists and is unique up to a unique isomorphism.)
The fundamental group $\pi_1(X,x)$ can be identified with the automorphism group of $\pitilde_x:\Xtilde_x\to X$.
Closed subgroups of $\pi_1(X,x)$ then correspond to ``sub-pro-\'etale covers'' of $\Xtilde_x \to X$, that is,
to factorisations $\Xtilde_x \rightarrow X' \rightarrow X$ of~$\pitilde_x$
where~$X'$ is a connected pro-\'etale cover of~$X$.

The morphism $\Xbar \rightarrow X$ is a connected pro-\'etale cover.  Moreover $\Xbar$ is canonically endowed with a $\kbar$\nobreakdash-point above~$x$.
Hence
$\pitilde_x:\Xtilde_x \to X$ factors canonically through~$\Xbar$, so that~$\Xtilde_x$ can be viewed as the universal cover of~$\Xbar$ at~$x$.
In general, there need not exist a pro-cover of~$X$ which, after extension of scalars from~$k$ to~$\kbar$,
becomes isomorphic to the universal cover of~$\Xbar$ at~$x$.  Such a pro-cover of~$X$ exists if and only if
the natural map $\pi_1(X,x) \rightarrow G_k$ admits a (continuous, as will always be implied) section.
More precisely, let $s \colon G_k \rightarrow \pi_1(X,x)$ be a section. Let us denote
by $\pi_s \colon X_s \rightarrow X$ the sub-pro-\'etale cover of $\pitilde_x \colon \Xtilde_x \to X$ which corresponds to the subgroup $s(G_k) \subseteq \pi_1(X,x)$.
The scheme $X_s \otimes_k \kbar$ is then canonically isomorphic to $\Xtilde_x$
over~$\Xbar$.  In particular~$X_s$ is a directed inverse limit of \'etale covers of~$X$ each of which is geometrically connected over~$k$.
If~$s$ is the section associated (up to conjugacy by~$\pi_1(\Xbar, x)$) to a rational point $a \in X(k)$, then $X_s(k)\neq\emptyset$; more precisely,
the set $X_s(k)$ then contains a canonical lifting $\atilde \in X_s(k)$ of~$a$.

Let~$\Lambda$ denote the ring~$\Z/N\Z$ for some integer $N \geq 1$ which is invertible in~$k$.
For $m \in \Z$, we denote by $\Lambda(m)$ the $m$\nobreakdash-th Tate twist of~$\Lambda$ (so $\Lambda(1)=\mmu_N$).

\begin{defn}\label{defn2.1}
We say that the variety~$X$ \emph{is a \kpi{}} if $\varinjlim H^i(Y, \Lambda)=0$
for all $i\ge 1$ and all $N \geq 1$ such that~$N$ is invertible in~$k$.  The direct limit
ranges over all factorisations $\Xtilde_x \rightarrow Y \rightarrow X$ of $\pitilde_x \colon \Xtilde_x \rightarrow X$ where~$Y$ is finite over~$X$ (for some fixed $x \in X(\kbar)$).
This property is purely geometric: it only depends on the scheme~$\Xbar$.
\end{defn}

An extended discussion of the notion of \kpi{} varieties may be found in \cite[Appendix~A]{Stix}.
The equivalence between Definition~\ref{defn2.1} and \cite[Definition~A.1.2]{Stix} is established in \cite[Proposition~A.3.1]{Stix}.

\begin{examples}
(i) Smooth proper curves of genus $g \geq 1$ as well as smooth affine curves are~\kpi's.

(ii) If~$X$ is a \kpi{} variety, any \'etale cover of~$X$ is also a \kpi{}.

(iii) According to Artin~\cite{artinsga}, on any smooth variety over a field of characteristic~$0$,
the~\kpi{} open subsets form a basis of the Zariski topology.
\end{examples}

Under the assumption that~$X$ is a \kpi{},
we shall now associate, to any section $s \colon G_k \rightarrow \pi_1(X,x)$, a ``cycle class'' in the \'etale cohomology
group with compact supports $H^{2d}_c(X,\Lambda(d))$, where $d=\dim(X)$.

\begin{defn}
(i) If~$\Vhat$ is a variety over~$k$ and $V \subseteq \Vhat$ is a dense open subset with complement~$D$, we shall denote
the \'etale cohomology group
$H^m(\Vhat,j_!\Lambda(n))$, where~$j$ stands for the inclusion $j \colon V \hookrightarrow \Vhat$,
by $H^m(\Vhat,D,\Lambda(n))$.  When~$\Vhat$ is a proper variety,
this group depends only on~$V$ and is usually denoted $H^m_c(V,\Lambda(n))$.

(ii) If $s \colon G_k \rightarrow \pi_1(X,x)$ is a section, we shall say that
an \'etale cover $\pi \colon Y \rightarrow X$ \emph{appears in} $\pi_s \colon X_s \rightarrow X$
if the morphism~$\pi_s$ factors through~$\pi$.  Given an \'etale cover $Y \rightarrow X$ appearing in~$\pi_s$, we shall say that an \'etale cover $Z \rightarrow Y$ \emph{appears in~$\pi_s$}
if the composition $Z \rightarrow X$ appears in~$\pi_s$.
A~property depending on an \'etale cover $Y \rightarrow X$ will be said to
\emph{hold if $Y \rightarrow X$ appears high enough in~$\pi_s$} if there exists an \'etale cover $Y_0 \rightarrow X$ appearing in~$\pi_s$ such that all
\'etale covers $Y \rightarrow X$ which appear in~$\pi_s$ and which factor through~$Y_0$ satisfy the given property.
Finally, given an \'etale cover $Y \rightarrow X$ appearing in~$\pi_s$, we shall say that a property depending on an \'etale cover $Z \rightarrow Y$
\emph{holds if $Z \rightarrow Y$ appears high enough in~$\pi_s$} if there exists an \'etale cover $Z_0 \rightarrow Y$ appearing in~$\pi_s$
such that all \'etale covers $Z \rightarrow Y$ which appear in~$\pi_s$ and which factor through~$Z_0$ satisfy the given property.
\end{defn}

\begin{prop}\label{prop2.2}
Let~$\Vhat$ be a variety over~$k$ and $V \subseteq \Vhat$ be a dense open subset, with complement $D=\Vhat \setminus V$.
Let~$X$ be a geometrically connected variety over~$k$.
Let $x \in X(\kbar)$.
Let $s \colon G_k \rightarrow \pi_1(X,x)$ be a section of the natural map $\pi_1(X,x) \rightarrow G_k$.
Assume~$X$ is a \kpi.
Then for any \'etale cover $\pi \colon Y \rightarrow X$ appearing high enough in~$\pi_s$, the following conditions hold: for any $m \geq 0$ and any $n \in \Z$,
the pullback map
$$
H^m(\Vhat, D, \Lambda(n)) \longrightarrow H^m(\Vhat\times Y, D\times Y, \Lambda(n))
$$
induced by the first projection $\Vhat \times Y \rightarrow \Vhat$ is injective and its image is equal to the image of the pullback map
$$
(1 \times \pi)^\star \colon
H^m(\Vhat\times X, D\times X, \Lambda(n))
\longrightarrow
H^m(\Vhat\times Y, D\times Y, \Lambda(n))\rlap{\text{.}}
$$
\end{prop}

\begin{proof}
Let~$j$ denote the inclusion $j \colon V \times Y \hookrightarrow \Vhat \times Y$.
Consider the Leray spectral sequence for the first projection $p \colon \Vhat \times Y \rightarrow \Vhat$ and the \'etale sheaf
$j_! \Lambda(n)$:
\begin{equation}
\label{lasuitespectraleleray}
E^{a,b}_2(Y):=H^a(\Vhat, R^bp_\star (j_! \Lambda(n))) \Longrightarrow
E^{a+b}(Y):=H^{a+b}(\Vhat\times Y,j_! \Lambda(n))
\end{equation}
We are interested in determining the kernel and the image of the natural maps $E^{m,0}_2(Y) \rightarrow E^m(Y)$
under the assumption that~$Y$ appears high enough in~$\pi_s$,
since
 $E_2^{m,0}(Y)=H^m(\Vhat,D,\Lambda(n))$
and $E^m(Y)=H^m(\Vhat \times Y, D \times Y, \Lambda(n))$.
 Let us denote by $(F^iE^m(Y))_{i \geq 0}$ the descending filtration
on  $E^m(Y)$ determined by~(\ref{lasuitespectraleleray}), so that
 $F^iE^m(Y)/F^{i+1}E^m(Y)=E^{i,m-i}_\infty(Y)$.

\begin{lem}
\label{lemmeexisteZ}
Let $Y \rightarrow X$ be an \'etale cover appearing in~$\pi_s$.
If $Z \rightarrow Y$ is an \'etale cover which appears high enough in~$\pi_s$,
then for every $a \geq 0$ and every $b \geq 1$, the natural map $E^{a,b}_2(Y) \rightarrow E^{a,b}_2(Z)$ is zero.
\end{lem}

\begin{proof}
The \'etale sheaf $R^bp_\star(j_!\Lambda(n))$ on~$\Vhat$ coincides with the extension by zero from~$V$ to~$\Vhat$ of the pullback, by the structure morphism $V \rightarrow \Spec(k)$,
of the \'etale sheaf defined by the $G_k$\nobreakdash-module $H^b(\Ybar,\Lambda(n)) = H^b(\Ybar,\Lambda) \otimes_{\Lambda} \Lambda(n)$,
where $\Ybar=Y \otimes_k \kbar$.  As a consequence we need only check that there exists an \'etale cover $Z \rightarrow Y$ appearing in~$\pi_s$
 such that for every $b \geq 1$, the pullback map $H^b(\Ybar,\Lambda) \rightarrow H^b(\Zbar,\Lambda)$ is zero.
Since~$X$ (and therefore~$Y$) is a \kpi{}, there exists an \'etale cover $\Zbar \rightarrow \Ybar$ such that for every $b \geq 1$,
the pullback map $H^b(\Ybar,\Lambda) \rightarrow H^b(\Zbar,\Lambda)$ is zero.  After replacing~$\Zbar$ with an \'etale cover of~$\Zbar$, we may assume that the composition
$\Zbar \rightarrow \Ybar \rightarrow Y$ is Galois.  This implies that there exists an \'etale cover $Z \rightarrow Y$ which appears in~$\pi_s$ and whose base change to~$\Ybar$
is $\Zbar \rightarrow \Ybar$.
\end{proof}

It follows from Lemma~\ref{lemmeexisteZ} that for any $Y \rightarrow X$ appearing in~$\pi_s$ and any $i<m$, if~$Z \rightarrow Y$ appears high enough in~$\pi_s$
then the image of the natural map $F^iE^m(Y) \rightarrow F^iE^m(Z)$ is contained in $F^{i+1}E^m(Z)$.  By iterating, we see that
if $Y \rightarrow X$ appears high enough in~$\pi_s$, then the image of the natural map $E^m(X) \rightarrow E^m(Y)$ is contained in $F^mE^m(Y)$.
Now $F^mE^m(Y)$ is the image of the map $E^{m,0}_2(Y) \rightarrow E^m(Y)$.  Hence the second assertion of the proposition is established.

To prove that the map $E^{m,0}_2(Y) \rightarrow E^m(Y)$ is injective if~$Y$ appears high enough in~$\pi_s$, it suffices to check
(by iteration again)
that for every $p \geq 2$, the natural map $E^{m,0}_p(Y) \rightarrow E^{m,0}_{p+1}(Y)$ is injective if~$Y$ appears high enough in~$\pi_s$.  The kernel
of the latter map is a quotient of $E^{m-p,p-1}_p(Y)$.  Hence the result again follows from Lemma~\ref{lemmeexisteZ} (since $p-1 \geq 1$).
\end{proof}

\begin{constthm}\label{thm2.5}
Let $X$ be a smooth, geometrically connected and separated variety of dimension~$d$ over~$k$.  We assume that~$X$ is a~\kpi.
Let $x \in X(\kbar)$ and let~$s$ be a section of the natural map $\pi_1(X,x) \rightarrow G_k$.
Let $N\geq 1$ be an integer invertible in~$k$.  Put $\Lambda=\Z/N\Z$.
To~$s$ we associate a class $\alpha(\Lambda) \in H^{2d}_c(X,\Lambda(d))$.  More generally, to~$s$ and to any \'etale cover $Y \rightarrow X$ appearing in~$\pi_s$, we associate
a class $\alpha(\Lambda,Y) \in H^{2d}_c(Y,\Lambda(d))$.

All these classes are compatible as~$Y$ and~$\Lambda$ vary, in the sense that for any prime number~$\ell$ which is invertible in~$k$, they define an element~$\alphatilde$ of the inverse limit
$$
\varprojlim_{m \geq 1} \varprojlim_{Y \rightarrow X}  H^{2d}_c(Y,\Z/\ell^m\Z(d))
$$
where the second limit ranges over all factorisations $X_s \rightarrow Y \rightarrow X$ of~$\pi_s$ with~$Y$ finite over~$X$, and where the transition morphism associated
to $Z \rightarrow Y$
is the trace map $H^{2d}_c(Z,\Z/\ell^m\Z(d)) \rightarrow H^{2d}_c(Y,\Z/\ell^m\Z(d))$.

These classes have degree~$1$, in the sense that for any \'etale cover $Y \rightarrow X$ appearing in~$\pi_s$, the image of $\alpha(\Lambda,Y)$ in $H^{2d}_c(\Ybar,\Lambda(d))=\Lambda$ is equal to $1 \in \Lambda$.

Moreover, if~$s$ is the section associated (up to conjugacy) to a rational point $a \in X(k)$, then~$\alphatilde$ is the cycle class of the rational point $\atilde \in X_s(k)$.
\end{constthm}
\begin{proof}
Let $X \subseteq \Xhat$ be a compactification of~$X$ and let $D=\Xhat \setminus X$.
Since~$X$ is smooth and separated over~$k$, the diagonal embedding $X \hookrightarrow \Xhat \times X$ defines
a class $[\Delta] \in H^{2d}(\Xhat \times X, D \times X, \Lambda(d))$.
According to Proposition~\ref{prop2.2}, if $\pi \colon Z \rightarrow X$ is an \'etale cover
which appears high enough in~$\pi_s$, the inverse image $(1\times \pi)^\star[\Delta]$ of~$[\Delta]$ in $H^{2d}(\Xhat \times Z, D \times Z, \Lambda(d))$
comes, by pullback, from a unique element of $H^{2d}(\Xhat,D,\Lambda(d))=H^{2d}_c(X,\Lambda(d))$.  It is this element which we denote $\alpha(\Lambda)$.
It~does not depend on the choice of~$Z$.  It maps to~$1$ in $H^{2d}_c(\Xbar,\Lambda(d))=\Lambda$ because for any $z \in Z(\kbar)$,
the class
$(1 \times \pi)^\star[\Delta]$ obviously
maps to~$1$ by $H^{2d}(\Xhat\times Z,D\times Z, \Lambda(d)) \rightarrow H^{2d}(\Xhat \times \{z\}, D \times \{z\},\Lambda(d))=H^{2d}_c(\Xbar,\Lambda(d))=\Lambda$.

For any \'etale cover $Y \rightarrow X$ appearing in~$\pi_s$, the section~$s$ induces a section of the natural map $\pi_1(Y,y)\rightarrow G_k$ for some $y \in Y(\kbar)$ above~$x$.
By applying the previous construction to~$Y$ and to this new section, we therefore obtain a degree~$1$ class $\alpha(\Lambda,Y) \in H^{2d}_c(Y,\Lambda(d))$.

That these classes are compatible if~$\Lambda$ and~$Y$ are allowed to vary follows from the fact that
the construction of~$\alpha(\Lambda)$ can be carried out with a given \'etale cover $\pi \colon Z \rightarrow X$ as soon as it appears high enough in~$\pi_s$.

Suppose now that~$s$ is associated to a rational point $a \in X(k)$.  To prove that~$\alphatilde$ is the cycle class of~$\atilde$, it suffices
to prove that $\alpha(\Lambda)$ is the cycle class of~$a$ (one can then apply the argument to every \'etale cover $Y \rightarrow X$ which appears in~$\pi_s$).
Let $\pi \colon Z \rightarrow X$ be as in the construction of~$\alpha(\Lambda)$.
Let $a_Z \in Z(k)$ denote the image of $\atilde \in X_s(k)$.  Restriction
to $\Xhat \times \{a_Z\} \subseteq \Xhat \times Z$ defines a map
$H^{2d}(\Xhat \times Z, D \times Z, \Lambda(d)) \rightarrow H^{2d}(\Xhat,D,\Lambda(d))$ which is a retraction of the natural map in the other direction.
Therefore $\alpha(\Lambda)$ is the image of $(1 \times \pi)^\star[\Delta]$ by this map.  Now $(1 \times \pi)^\star[\Delta]$ is the cycle class of the graph of~$\pi$;
hence its restriction to $\Xhat \times \{a_Z\}$ is the cycle class of $\pi(a_Z)=a$.
\end{proof}

\begin{rmk}
\label{rmk2.7}
Let us not assume that~$X$ is a $K(\pi,1)$.
Let~$s$ be a section of the natural map $G_{k(X)}\rightarrow G_k$,
where~$G_{k(X)}$ denotes the absolute Galois group of~$k(X)$.
Then~$s$ determines (up to conjugacy) a section of $\pi_1(U,u) \rightarrow G_k$ for any dense open $U \subseteq X$ and any $u \in U(\kbar)$.
In particular, by choosing~$U$ to be small enough, one may assume that~$U$ is a $K(\pi,1)$ and apply Theorem~\ref{thm2.5} to produce a well-defined class $\alpha(\Lambda,U) \in H^{2d}_c(U,\Lambda(d))$.
The image of~$\alpha(\Lambda,U)$ in $H^{2d}_c(X,\Lambda(d))$ does not depend on the choice of~$U$.  We shall again denote it by~$\alpha(\Lambda)$.
\end{rmk}

\section{Algebraicity of cycle classes over \texorpdfstring{$p$-adic}{p-adic} fields}
\label{sectionalgebraicity}

Building upon earlier results of Pop, Koenigsmann~\cite{Ko} showed that if~$X$ is a smooth and geometrically connected proper curve over a $p$\nobreakdash-adic field~$k$,
then every section~$s$ of the natural map $G_{k(X)} \rightarrow G_k$
determines a unique rational point $a\in X(k)$ (in the sense that the section of $\pi_1(X,x) \rightarrow G_k$ induced by~$s$ is associated to~$a$).
Koenigsmann's proof is model-theoretic.
We observe here that Theorem~\ref{thm2.5} gives a geometric understanding of the ``abelian'' part of the rational point~$a$ (\emph{i.e.}, of~$a$ considered as a divisor of degree~$1$ on~$X$):

\begin{prop}\label{propkoenig}
Let~$X$ be a smooth proper geometrically connected curve over a $p$\nobreakdash-adic field~$k$.
Let~$s$ be a section of the natural map $G_{k(X)} \rightarrow G_k$. Let $\alpha \in H^2(X,\Zhat(1))$ denote the inverse limit of the classes $\alpha(\Z/n\Z) \in H^2(X,\mmu_n)$
constructed in Theorem~\ref{thm2.5} and Remark~\ref{rmk2.7}.
Then~$\alpha$ is the cycle class of a degree~$1$ divisor on~$X$ (uniquely determined up to linear equivalence).
\end{prop}

\begin{proof}
Multiplication by~$n$ on~$\Gm$ induces an exact sequence
$$
\xymatrix{
0 \ar[r] & \Pic(X)/n\Pic(X) \ar[r] & H^2(X,\mmu_n) \ar[r] & \Br(X)
}
$$
where $\Br(X)=H^2(X,\Gm)$.
According to Remark~\ref{rmk2.7}, the class $\alpha(\Z/n\Z)$ belongs to the image of the map $H^2_c(U,\mmu_n) \rightarrow H^2(X,\mmu_n)$
for every dense open $U \subseteq X$.  Hence its image in $\Br(X)$ belongs to the image of $H^2_c(U,\Gm) \rightarrow H^2(X,\Gm)$ for every
dense open $U \subseteq X$. In other words the image of $\alpha(\Z/n\Z)$ in $\Br(X)$ belongs to the right kernel of the natural pairing
$\Pic(X) \times \Br(X) \rightarrow \Br(k)$.
By Lichtenbaum--Tate duality~\cite{Li} this right kernel is zero; hence finally $\alpha(\Z/n\Z) \in \Pic(X)/n\Pic(X)$.
On the other hand, the image of~$\alpha$
in $H^2(\Xbar, \Zhat(1))=\Zhat$ is equal to~$1$.  Therefore $\alpha$ belongs to the inverse image of $1 \in \Zhat$ by the degree map
$\varprojlim_{n \geq 1} \left(\Pic(X)/n\Pic(X)\right) \rightarrow \Zhat$,
which means, since~$k$ is $p$\nobreakdash-adic, that $\alpha$ is the image of a unique element of~$\Pic(X)$
(see~\cite[I.3.3]{Mil}).
\end{proof}

Let now~$X$ be a smooth proper geometrically connected curve of genus $g \geq 2$ over a $p$\nobreakdash-adic field~$k$.
It is an open question whether any section $s \colon G_k \rightarrow \pi_1(X,x)$
coincides, up to conjugacy, with
the section associated to a rational point of~$X$
($p$\nobreakdash-adic analogue of Grothendieck's section conjecture).
For such a statement to hold, it is necessary that for any section~$s$, the cycle class constructed
in Theorem~\ref{thm2.5} be algebraic (\emph{i.e.}, be the cycle class of an algebraic cycle).
Proposition~\ref{propkoenig} (and more generally Koenigsmann's theorem)
asserts that such is the case if~$s$
admits a lifting to $G_{k(X)}$.
In Corollary~\ref{coralphaalg} below we prove that such is also the case in general
provided one restricts attention to cycle classes in $\ell$\nobreakdash-adic \'etale cohomology with $\ell \neq p$.
From this it follows (Corollary~\ref{corstix})
that the existence of a section $s \colon G_k \rightarrow \pi_1(X,x)$
implies the existence, on~$X$, of a divisor whose degree is a power of~$p$
(in other words, the index~$I$ of~$X$ is a power of~$p$).

Corollary~\ref{corstix} was already
known, see Stix~\cite{Stix2} (moreover, Corollary~\ref{coralphaalg} may in fact be deduced from Corollary~\ref{corstix}).
Quite generally, Lichtenbaum~\cite{Li} proved that the period~$P$ and the index~$I$
of a curve~$X$ as above over a $p$\nobreakdash-adic field~$k$
satisfy the divisibility relations $P \divise I \divise 2P$ and that moreover, if $I=2P$,
then $(g-1)/P$ is an odd integer.
If one believes that every section $s \colon G_k \rightarrow \pi_1(X,x)$ comes from a rational point, then one
should have $I=P=1$ as soon as a section exists.
Stix~\cite[Theorem~2]{Stix2} uses Lichtenbaum's result to show, assuming the existence of a section, that~$I=P$ if $p>2$
and that in any case~$I$ and~$P$ are powers of~$p$.
The proof we give below for Corollary~\ref{coralphaalg} and Corollary~\ref{corstix},
however, does not appeal to Lichtenbaum's result; it is based on a study of the cycle
classes associated to sections of $\pi_1(X,x)\rightarrow G_k$.

We henceforth denote by~$X$ a smooth proper geometrically connected \kpi{} variety of dimension~$d$ over a
field~$k$.
For any integers~$m$, $n$ and~$N$, we let $(F^iH^m(X,\Lambda(n)))_{i \geq 0}$, where  $\Lambda=\Z/N\Z$, denote the descending filtration
on $H^m(X,\Lambda(n))$ determined by the Leray spectral sequence for the structure
morphism $\epsilon \colon X \rightarrow \Spec(k)$ and the \'etale sheaf $\Lambda(n)$.
Let~$\ell$ be a prime number invertible in~$k$.
We set
$H^m(X, \Zl(n))=\varprojlim_{q \geq 1} H^m(X, \Z/\ell^q\Z(n))$,
$F^iH^m(X,\Zl(n))=\varprojlim_{q\geq 1} F^iH^m(X,\Z/\ell^q\Z(n))$
and we shall write~$\Ql$ coefficients to denote these $\Zl$\nobreakdash-modules tensored with~$\Ql$.
Any section $s \colon G_k \rightarrow \pi_1(X,x)$ induces
a retraction of $\epsilon^\star \colon H^{2d}(k,\Zl(d)) \rightarrow H^{2d}(X,\Zl(d))$ according
to Proposition~\ref{prop2.2} applied to $V=\Vhat=\Spec(k)$.  When a section~$s$ is given,
we denote this retraction by $r \colon H^{2d}(X,\Zl(d)) \rightarrow H^{2d}(k,\Zl(d))$
and we let $\alpha \in H^{2d}(X,\Zl(d))$ denote the cycle class of~$s$ (given by Theorem~\ref{thm2.5}).
We keep the notation~$\pi_s$ from~\textsection\ref{section2}.

\begin{prop}\label{prop3.2}
Let $s: G_k\to \pi_1(X)$ be a section. Assume
\begin{itemize}
 \item[(i)] $F^1H^{2d}(Y, \Ql(d))=F^{2d}H^{2d}(Y, \Ql(d))$ for all $Y\to X$ appearing in~$\pi_s$;
\item[(ii)] $H^{2d}(k, \Zl(d))$ is a finitely generated free $\Z_\ell$\nobreakdash-module.
\end{itemize}
Then $r(\alpha)=0$.
\end{prop}

\begin{proof}
Let $\pi \colon Y \rightarrow X$ be an \'etale cover appearing in~$\pi_s$, and let~$\delta$ denote its degree.
Let $\alpha_Y \in H^{2d}(Y,\Zl(d))$ be the cycle class of~$s$ on~$Y$ (see Theorem~\ref{thm2.5}).
The retraction~$r$ factors as $r=r_Y \circ \pi^\star$
where $r_Y \colon H^{2d}(Y,\Zl(d)) \rightarrow H^{2d}(k,\Zl(d))$
is the retraction of $(\epsilon \circ \pi)^\star \colon H^{2d}(k,\Zl(d)) \rightarrow H^{2d}(Y,\Zl(d))$
obtained by applying Proposition~\ref{prop2.2} to~$Y$ instead of~$X$.

Let $\beta = \delta\alpha_Y-\pi^\star\alpha \in H^{2d}(Y,\Zl(d))$.  Since the image of~$\alpha$ (resp.~$\alpha_Y$)
in $H^{2d}(\Xbar,\Zl(d))=\Zl$ (resp.~in $H^{2d}(\Ybar,\Zl(d))=\Zl$) is equal
to~$1$, one has $\beta \in F^1H^{2d}(Y,\Zl(d))$.  Hence, by~(i), the existence of an integer $N>0$ such that
$N\beta \in F^{2d}H^{2d}(Y,\Zl(d))$.  Now $F^{2d}H^{2d}(Y,\Zl(d))=\pi^\star \epsilon^\star H^{2d}(k,\Zl(d))$,
so that there exists $c \in H^{2d}(k,\Zl(d))$ such that $N\beta=\pi^\star \epsilon^\star c$.
Let us apply~$\pi_\star$ and then~$r$ to this equality.  One has $\pi_\star\beta=0$ since $\pi_\star \alpha_Y=\alpha$;
hence we find $\delta c=0$.  By~(ii) this implies $c=0$, so that $N\beta=0$.  It follows
that $N r_Y(\beta)=0$ and then, by~(i), that $r_Y(\beta)=0$.
In view of the equality $r(\alpha)=r_Y(\pi^\star(\alpha))$, the vanishing of $r_Y(\beta)$ means
that $r(\alpha)=\delta r_Y(\alpha_Y)$.  In particular~$r(\alpha)$ is divisible by~$\delta$.

The integer~$\delta$ can be chosen
at the beginning of the argument
 to be divisible by arbitrarily large powers of~$\ell$.
Therefore $r(\alpha)$ is infinitely divisible in $H^{2d}(k,\Zl(d))$, which, by~(ii), finally implies $r(\alpha)=0$.
\end{proof}

\begin{cor} \label{cor3.3}
Let~$X$ be a smooth proper geometrically connected curve of genus $g\ge 1$
over a $p$\nobreakdash-adic field~$k$. Let $s \colon G_k\rightarrow
\pi_1(X,x)$ be a section, and~$\ell$ be a prime number different from~$p$.
Then $r(\alpha)=0$.
\end{cor}

\begin{proof}
Condition~(ii) in Proposition~\ref{prop3.2} is satisfied since $H^2(k, \Zl(1))=\Z_\ell$.
Condition~(i) amounts to the cohomology group $H^1(k,V)$ being zero, where $V=H^1(\bar X,\Ql(1))$.  There are several ways to see this. 

\newcommand{\fibrespec}{{\bar S}}
One may apply Deligne's theory of weights. Let~$k^u$ be the maximal unramified extension of~$k$, let~$R$ (resp.~$R^u$) be the ring of integers of~$k$ (resp.~$k^u$),
let~$\sX$ be a proper regular model of~$X$ over~$R$,
and let~$\fibrespec$ denote the geometric special fiber of~$\sX$.
Put $\sX^u=\sX \otimes_R R^u$ and $X^u=X \otimes_k k^u$.
The eigenvalues of the geometric Frobenius acting on $H^0(k^u,V)$
have weights in $\{-2,-1\}$ according to~\cite[Exp.~IX, Cor.~4.4]{sga71}.
The kernel of the natural map $H^2(X^u,\Ql(1)) \rightarrow
H^3_{\fibrespec}(\sX^u,\Ql(1))$ has dimension~$1$ and weight~$0$ (since it is
the cokernel of the intersection matrix of~$\fibrespec$ tensored
with~$\Ql$). On the other hand, $H^3_{\fibrespec}(\sX^u,\Ql(1))$ has weights
in~$\{1,2\}$; hence $H^2(X^u,\Ql(1))$ has weights in $\{0,1,2\}$ and~$0$
appears only with multiplicity~$1$.  Now $H^2(X^u,\Ql(1))$ is an extension
of~$\Ql$ by $H^1(k^u,V)$.  Therefore $H^1(k^u,V)$ has weights in $\{1,2\}$.
We conclude that~$0$ does not appear as a weight of $H^i(k^u,V)$ for
any $i \in \{0,1\}$, which implies, by inflation-restriction, that $H^1(k,V)=0$.

One may also argue as follows.  Let~$J$ denote the Jacobian variety of~$X$.  Multiplication by~$n$ on~$J$ yields a short exact sequence
$$
\xymatrix{
0 \ar[r] & J(k) \otimeshat \Zl \ar[r] & H^1(k,T) \ar[r] & {\rm T}_\ell(H^1(k,J)) \ar[r] & 0\text{,}
}
$$
where $T=H^1(\Xbar,\Zl(1))$ (the notations ${\rm T}_\ell$ and $\otimeshat$ were defined in~\textsection\ref{secintro}).
According to Mattuck, there exists a subgroup of finite index in $J(k)$
isomorphic to~$R^g$ (see~\cite[I.3.3]{Mil}).
The group on the left-hand side of the above exact sequence is therefore finite.
On the other hand, Tate duality~\cite[I.3.4]{Mil} implies that the group on the right-hand side is
$\Zl$\nobreakdash-dual to~$J(k)$, so that it even vanishes.  Thus $H^1(k,T)$ is finite and hence $H^1(k,V)=0$.
\end{proof}

From Corollary~\ref{cor3.3} we now deduce the main corollary of Proposition~\ref{prop3.2}:

\begin{cor}\label{coralphaalg}
Let~$X$ be a smooth proper geometrically connected curve of genus $g\ge
2$ over a $p$\nobreakdash-adic field~$k$. Let $s \colon G_k\rightarrow
\pi_1(X,x)$ be a section, and~$\ell$ be a prime number different from~$p$. 
Let $\alpha \in H^2(X,\Zl(1))$ denote the cycle class of~$s$.
Then $\alpha$ is an algebraic cycle class, \emph{i.e.}, it belongs to the image of
the cycle class map $\Pic(X) \otimeshat \Zl\rightarrow H^2(X,\Zl(1))$.
\end{cor}

\begin{proof}
Kummer theory yields an exact sequence
\begin{equation}
\label{exsekum1}
\xymatrix{
0 \ar[r] & \Pic(X) \otimeshat \Zl \ar[r] & H^2(X,\Zl(1)) \ar[r] & T_\ell(\Br(X)) \ar[r] & 0\text{.}
}
\end{equation}
Since $T_\ell(\Br(X))$ is torsion-free, one may as well work with the exact sequence
\begin{equation}
\label{exsekum2}
\xymatrix{
0 \ar[r] & \Pic(X) \otimeshat \Ql \ar[r] & H^2(X,\Ql(1)) \ar[r] & V_\ell(\Br(X)) \ar[r] & 0\text{;}
}
\end{equation}
it is then enough to show that $\alpha \in H^2(X,\Ql(1))$ lies in the image of $\Pic(X) \otimeshat \Ql$.
Let $\omega \in H^2(X,\Ql(1))$ denote the cycle class of the canonical sheaf.
Then the class $\alpha-\omega/(2g-2)$ has degree~$0$, \emph{i.e.}, it vanishes in $H^2(\Xbar,\Ql(1))=\Ql$.
Therefore it belongs to the image of $\epsilon^\star \colon H^2(k,\Ql(1)) \rightarrow
H^2(X,\Ql(1))$, since $H^1(k,V)=0$ (in the notation of the proof of Corollary~\ref{cor3.3}).
Now Corollary~\ref{cor3.3} and Lemma~\ref{lemmaalphaomega} below imply that it also belongs to the kernel
of $r \colon H^2(X,\Ql(1)) \rightarrow H^2(k,\Ql(1))$.  As a consequence, it vanishes; hence $\alpha=\omega/(2g-2)$
in $H^2(X,\Ql(1))$, which
establishes the corollary.
\end{proof}

\begin{lem}\label{lemmaalphaomega}
Let~$X$ be a smooth proper geometrically connected curve of genus $g \geq 1$ over a field~$k$.
Let $s \colon G_k \rightarrow \pi_1(X,x)$ be a section.  Let $\Lambda=\Z/N\Z$, where~$N$ is invertible in~$k$.
We still denote by $\alpha \in H^2(X,\Lambda(1))$ the cycle class of~$s$, by $\omega \in H^2(X,\Lambda(1))$ the cycle class
of the canonical sheaf, and by $r \colon H^2(X,\Lambda(1)) \rightarrow H^2(k,\Lambda(1))$ the retraction induced by~$s$.
Then $r(\alpha+\omega)=0$.
\end{lem}

\begin{proof}
Let $\pi \colon Y \rightarrow X$ be an \'etale cover appearing high enough in~$\pi_s$.
Denote by $p\colon X \times Y \rightarrow X$ and $q \colon X \times Y \rightarrow Y$ the two projections,
and by $i \colon Y \rightarrow X \times Y$ the embedding $\pi \times 1$.
Let $\gamma \in H^2(X \times Y,\Lambda(1))$ be the cycle class of the graph of~$\pi$.
The adjunction formula reads $i^\star (\gamma + p^\star \omega + q^\star \omega_Y)=\omega_Y$, where $\omega_Y \in H^2(Y,\Lambda(1))$
denotes the cycle class of the canonical sheaf of~$Y$.  It follows that $i^\star \gamma = -\pi^\star \omega$.
Now by the definition of~$\alpha$, one has $p^\star\alpha = \gamma$, hence $\pi^\star\alpha=i^\star \gamma$.
Thus $\pi^\star (\alpha+\omega)=0$, which proves the lemma.
\end{proof}

\newcommand{\citestix}{\cite[Theorem~2]{Stix2}}
\begin{cor}[see Stix~\citestix]\label{corstix}
Let~$X$ be a smooth proper geometrically connected curve of genus $g\ge 2$
over a $p$\nobreakdash-adic field~$k$. Assume the natural map $\pi_1(X,x) \rightarrow G_k$ admits a section.
Then the index of~$X$ is a power of~$p$ (\emph{i.e.}, there exists a divisor on~$X$ whose degree is a power of~$p$).
\end{cor}

\begin{proof}
It suffices to show that for any prime number~$\ell$ different from~$p$,
there exists a divisor on~$X$ whose degree is prime to~$\ell$.  Fix a prime $\ell\neq p$.
Choose a section $s \colon G_k \rightarrow \pi_1(X,x)$.  Denote by $\alpha \in H^2(X,\Zl(1))$ the $\ell$\nobreakdash-adic cycle class
of~$s$ and by $\alpha_1 \in H^2(X,\Z/\ell\Z(1))$ the mod~$\ell$ cycle class of~$s$.
According to Corollary~\ref{coralphaalg}, $\alpha$ is algebraic; therefore so is $\alpha_1$.
In other words $\alpha_1 \in \Pic(X)/\ell\Pic(X)$.  Now any divisor on~$X$ whose class
in $\Pic(X)/\ell\Pic(X)$ equals~$\alpha_1$
has degree congruent to~$1$ modulo~$\ell$,
since the image of $\alpha_1$ in $H^2(\Xbar,\Z/\ell\Z(1))=\Z/\ell\Z$ is equal to~$1$.
\end{proof}

\begin{rmks}\label{rmks37}
(i) Let~$X$ and~$s$ be as in Corollary~\ref{coralphaalg}.  Let~$\ell$ be a prime number (possibly equal to~$p$).
Lichtenbaum--Tate duality asserts that $\Pic(X) \otimeshat \Ql$ is a maximal totally isotropic subspace of $H^2(X,\Ql(1))$
with respect to the cup-product pairing (which takes values in $H^4(X,\Ql(2))=\Ql$).
If $\ell \neq p$ then this subspace is simply the line generated by~$\omega$; hence in this case,
in order to prove that a class of $H^2(X,\Ql(1))$
is algebraic, it suffices to check that it is orthogonal to~$\omega$.   This is precisely what we do in the proof
of Corollary~\ref{coralphaalg}. Indeed the condition $\omega \cup \alpha=0$ is equivalent to $r(\omega)=0$.
More generally it follows from the definition of~$r$ that $r(x)=x \cup \alpha$ for any $x \in H^2(X,\Ql(1))$.

(ii) One cannot expect the statements of Corollaries~\ref{coralphaalg} and~\ref{corstix} to hold in the case of genus~$1$ curves without the assumption that $\ell\neq p$
(and in fact they fail).  However, for $\ell \neq p$, these statements do hold for genus~$1$ curves.  Indeed,
 in~\cite[Theorem~2]{Stix2}, the genus~$g$ is only assumed to
be~$\geq 1$.
Jakob Stix points out to us that the proof of Corollary~\ref{coralphaalg} given above can be adapted
to cover uniformly the $g=1$ and $g>1$ cases in the following way. Let $\lambda \in \Pic(X) \otimeshat \Ql$ be the class which has degree~$1$
(when $g>1$, it is given by $\lambda=\omega/(2g-2)$).
Then $\lambda-\alpha \in H^2(X,\Ql(1))$ has degree~$0$, hence it belongs to the image of~$\epsilon^\star$, hence $(\lambda-\alpha) \cup (\lambda-\alpha)=0$.
In view of Corollary~\ref{cor3.3} and of Remark~\ref{rmks37}~(i), this implies that $\alpha \cup \lambda=0$ and hence that $\alpha=\lambda$
in $H^2(X,\Ql(1))$.

(iii) Let~$X$, $s$ and~$\ell$ be as in Corollary~\ref{coralphaalg}, and let~$\sX$ denote a proper regular model of~$X$ over the ring of integers of~$k$.
It follows from Corollary~\ref{coralphaalg} that the $\ell$\nobreakdash-adic cycle class $\alpha \in H^2(X,\Zl(1))$ admits liftings to $H^2(\sX,\Zl(1))$.
Indeed the restriction map $\Pic(\sX) \otimeshat \Zl \rightarrow \Pic(X) \otimeshat \Zl$ is surjective, and its composition with the cycle class
map $\Pic(X) \otimeshat\Zl \rightarrow H^2(X,\Zl(1))$ factors through $H^2(\sX,\Zl(1))$.
Now if the $p$\nobreakdash-adic analogue of Grothendieck's section conjecture holds true,
then~$\alpha$ should even admit a canonical lifting to $H^2(\sX,\Zl(1))$; to wit, if~$s$ comes from a rational point $a \in X(k)$, then the cycle class in $H^2(\sX,\Zl(1))$
of the closure of~$a$ in~$\sX$ is a lifting of~$\alpha$.  Starting from an
arbitrary section~$s$, we are unable to construct such a canonical lifting in general.
We can do it only under the assumption that none of the irreducible components
of the geometric special fiber of~$\sX$ is rational (using a variant of the arguments employed in the proof of Theorem~\ref{thm2.5}).

(iv) Corollary~\ref{corstix} has an analogue over the real numbers.  Namely, many authors have
noticed that the following statement holds: let~$X$ be a smooth proper geometrically connected curve of genus $g\geq 1$ over the field~$k=\R$ of real numbers;
assume $\pi_1(X,x)\rightarrow G_k$ admits a section; then $X(k)\neq \emptyset$. (This is the ``real analogue of the weak section conjecture''. See~\cite[Appendix~A]{Stix2} for a summary.)
To the best of our knowledge, all previous proofs rely on nontrivial theorems in real algebraic geometry (by Witt, Artin, Verdier, or Cox).
Here we remark that Tsen's theorem (according to which the Brauer group of a complex curve vanishes) suffices.  Indeed it implies that $\Br(X)$ has exponent~$2$,
and hence that ${\rm T}_2(\Br(X))=0$.  Now let $s \colon G_k \rightarrow \pi_1(X,x)$ be a section.
According to~(\ref{exsekum1}),
the cycle class of~$s$ in $H^2(X,\Z_2(1))$
belongs to $\Pic(X) \otimeshat \Z_2$. As in the proof of Corollary~\ref{corstix} we conclude that~$X$ has odd index and therefore
that $X(k)\neq\emptyset$.
\end{rmks}

\section{Nilpotent completion and liftings of cycle classes of sections}
\label{section4}
\subsection{Beilinson's construction}\label{sub3.1}

Let~$X$ be a complex manifold and let $a \in X$.  Denote by~$I$ the augmentation ideal of the group algebra $\Q[\pitop(X,a)]$.
For any $n \geq 1$, Beilinson constructed a complex of sheaves of $\Q$\nobreakdash-vector spaces on~$X^n$
whose $n$\nobreakdash-th hypercohomology group is canonically dual to $\Q[\pitop(X,a)]/I^{n+1}$.
His construction is described in~\cite[Section~3]{DelGon}.
We recall it briefly.
Consider the manifold $X \times X^n \times X$ with coordinates $(t_0,\dots,t_{n+1})$.  For $i \in \{0,\dots,n\}$, let~$A_i$ denote the submanifold defined
by $t_i=t_{i+1}$.  For $J \subseteq \{0,\dots,n\}$, let $A_J = \cap_{j \in J} A_j$ and let $\Q_{A_J}$ denote the direct image of the constant sheaf~$\Q$ by the
inclusion $A_J \hookrightarrow X \times X^n \times X$.
Let $\sB(n)$ denote the complex of sheaves on $X \times X^n \times X$
defined by
$\sB(n)^p=\bigoplus_{J \subseteq \{0,\dots,n\},\# J = p} \Q_{A_J}$
for $p \in \{0,\dots,n\}$
and $\sB(n)^p=0$ for all other~$p$; the differential $\sB(n)^p \rightarrow \sB(n)^{p+1}$ is the sum, over all~$J$ and all~$c$ such that $c\not\in J$,
of $(-1)^{\#\{m \in J;m<c\}}$ times the restriction map $\Q_{A_J} \rightarrow \Q_{A_{J \cup \{c\}}}$.
For $(b,a) \in X \times X$, let~$i_{ba}$ denote the closed immersion $X^n = \{b\} \times X^n \times \{a\} \hookrightarrow X \times X^n \times X$
and let $j_{ba}$ denote the open immersion
$X^n \setminus  i_{ba}^{-1}\left(\bigcup_{i=0}^n A_i\right) \hookrightarrow X^n$.
Set $\sB(n)_{ba}=i_{ba}^\star \sB(n)$.
If $b \neq a$, then $\sB(n)_{ba}$ is quasi-isomorphic to $j_{ba!}\Q$.  On the other hand, if $b=a$,
the defect of exactness of~$\sB(n)$ at $\sB(n)^n$ provides a map $\sB(n) \rightarrow \Q_{A_{\{0,\dots,n\}}}[-n]$
which in hypercohomology induces a map $\theta_{aa} \colon \H^n(X^n,\sB(n)_{aa}) \rightarrow \Q$.
Proposition~3.4 of \emph{loc.\ cit.}\ then asserts that the $\Q$\nobreakdash-vector space $\H^n(X^n,\sB(n)_{aa})$, endowed with~$\theta_{aa}$,
is canonically dual to $\Q[\pitop(X,a)]/I^{n+1}$, endowed with the map $\Q \rightarrow \Q[\pitop(X,a)]/I^{n+1}$ which sends~$1$ to~$1$.

\subsection{Replacing the base point by the universal cover}\label{sub3.2}
We first remark that when~$X$ is a~\kpi,
Beilinson's construction can be reformulated in terms of the (topological) universal cover of~$X$ at~$a$ instead of the point~$a$ itself.
Let~$X$ be as above.  For $a \in X$, let $\pitildetop_a \colon (\Xtildetop_a,\atilde) \rightarrow (X,a)$ denote the (topological) universal pointed cover of the pointed space $(X,a)$.
Its fiber above~$a$ is $\pitop(X,a)$.

\begin{defn}\label{defn3.1}
(i) We say that~$X$ is \emph{topologically a~\kpi} if $H^i(\Xtildetop_a,\Z)=0$ for all $i \geq 1$.

(ii) Let $\pi \colon Y \rightarrow X$ be a topological cover.  Let $p \colon Y \times X^n \times Y \rightarrow X \times X^n \times X$ denote
the map $\pi \times 1 \times \pi$.  We put $\sB(n)(Y)=p^\star \sB(n)$.
The defect of exactness of the complex $\sB(n)(Y)$ at $\sB(n)(Y)^n$ provides a map
$\sB(n)(Y) \rightarrow \iota_\star \Q[-n]$ where $\iota \colon Y \rightarrow Y \times X^n \times Y$ is the closed immersion $1 \times \pi^n \times 1$.
We denote by $\theta_\pi \colon \H^n(Y \times X^n \times Y, \sB(n)(Y)) \rightarrow \Q$ the linear form it induces in hypercohomology.
\end{defn}
\begin{prop}\label{propcorrectcohomology}
Assume~$X$ is topologically a~\kpi.  For any $a \in X$, the $\Q$\nobreakdash-vector space
$\H^n(X^n,\sB(n)_{aa})$, endowed with the linear form~$\theta_{aa}$,
is canonically isomorphic to
$\H^n(\Xtildetop_a \times X^n \times \Xtildetop_a, \sB(n)(\Xtildetop_a))$ endowed with~$\theta_{\pitildetop_a}$.
\end{prop}

\begin{proof}
Let $i \colon X^n \hookrightarrow \Xtildetop_a \times X^n \times \Xtildetop_a$
denote the closed immersion $\{\atilde\} \times 1 \times \{\atilde\}$.
The inverse image of $\sB(n)(\Xtildetop_a)$ by~$i$ is equal to $\sB(n)_{aa}$.
As a consequence, to establish the proposition
it suffices to check that the restriction map
$$\H^n(\Xtildetop_a \times X^n \times \Xtildetop_a, \sB(n)(\Xtildetop_a)) \longrightarrow\H^n(X^n, i^\star\sB(n)(\Xtildetop_a))$$
is an isomorphism.  For this, in view of the definition of~$\sB(n)$, it suffices to check that the restriction map
$$H^m(\Xtildetop_a \times X^n \times \Xtildetop_a, p^\star \Q_{A_J}) \longrightarrow H^m(X^n, i^\star p^\star \Q_{A_J})$$
is an isomorphism for any~$m$ and for any $J \varsubsetneq \{0,\dots,n\}$, where~$p$ is as in Definition~\ref{defn3.1}~(ii).
Now this is a direct consequence of the hypothesis that~$X$ is topologically a~\kpi{} together with the K\"unneth formula.
\end{proof}

As a consequence, the $\Q$\nobreakdash-vector space $\H^n(\Xtildetop_a \times X^n \times \Xtildetop_a, \sB(n)(\Xtildetop_a))$, endowed with~$\theta_{\pitildetop_a}$,
is canonically dual to $\Q[\pitop(X,a)]/I^{n+1}$, endowed with the map $\Q \rightarrow \Q[\pitop(X,a)]/I^{n+1}$ which sends~$1$ to~$1$.

\subsection{In the algebraic setting}\label{sub3.3}
Let~$X$ be a smooth geometrically irreducible separated variety of dimension~$d$ over a field~$k$ with separable closure~$\kbar$.  Let~$N$ be an integer invertible in~$k$ and let $\Lambda=\Z/N\Z$.
Fix $n \geq 1$.  For $J \subseteq \{0,\dots,n\}$, let $A_J \subseteq X \times X^n \times X$ be, as in~\textsection\ref{sub3.1}, the subvariety defined by the equations $t_j=t_{j+1}$ for $j\in J$,
and let $\Lambda_{A_J}$ denote the direct image of the constant \'etale sheaf~$\Lambda$ by the inclusion $A_J \hookrightarrow X \times X^n \times X$.
Let $\sB^e(n)$ (where~$e$ stands for ``\'etale'') denote the complex of \'etale sheaves on $X \times X^n \times X$ defined in the same way as~$\sB(n)$ was defined in~\textsection\ref{sub3.1}, except
that~$\Q_{A_J}$ is now replaced by~$\Lambda_{A_J}$.   
For $a \in X(\kbar)$, we denote by $\sB^e(n)_{aa}$ the inverse image of~$\sB^e(n)$ by the natural map $\Xbar^n = \{a\} \times \Xbar^n \times \{a\} \rightarrow X \times X^n \times X$,
where $\Xbar=X\otimes_k \kbar$.  The natural map $\sB^e(n) \rightarrow \Lambda_{A_{\{0,\dots,n\}}}[-n]$
induces a linear form $\theta_{aa} \colon \H^n(\Xbar^n,\sB^e(n)_{aa}) \rightarrow \Lambda$.
Finally, if $\pi \colon Y \rightarrow X$ is an \'etale cover, we denote
by $p \colon Y \times X^n \times Y \rightarrow X \times X^n \times X$ the map $\pi \times 1 \times \pi$.
Definition~\ref{defn3.1}~(ii) still makes sense. It yields
a linear form $\theta_\pi \colon \H^n(Y \times X^n \times Y, \sB^e(n)(Y)) \rightarrow \Lambda$, where $\sB^e(n)(Y)=p^\star \sB^e(n)$.

With these definitions in hand, we may formulate a statement analogous to Proposition~\ref{propcorrectcohomology}. Its proof, which we omit, is also entirely analogous.
\begin{prop}\label{prop3.4}
Assume~$X$ is a \kpi{}. For any $a \in X(\kbar)$, the $\Lambda$\nobreakdash-module
$\H^n(\Xbar^n,\sB^e(n)_{aa})$, endowed with $\theta_{aa}$, is canonically isomorphic
to
\begin{equation}
\varinjlim \H^n(\Ybar \times \Xbar^n \times \Ybar, \sB^e(n)(\Ybar))
\end{equation}
endowed with the linear form $\varinjlim \theta_\pi$.
The direct limit ranges over all factorisations $\Xtilde_a \rightarrow \Ybar \rightarrow \Xbar$
of $\pitilde_a \colon \Xtilde_a \rightarrow \Xbar$ with~$\Ybar$ finite over~$\Xbar$.
\end{prop}

We now turn to the $\Lambda$\nobreakdash-module dual to $\H^n(\Xbar^n,\sB^e(n)_{aa})$.

\begin{defn}\label{defn3.5}
Let $n \geq 1$.
We define an object~$\sC(n)$ in the bounded derived category of \'etale sheaves of $\Lambda$\nobreakdash-modules on $X \times X^n \times X$ by the formula
$$
\sC(n)=\tau_{\le n(2d-1) } Rj_\star\Lambda(nd)
$$
where~$j$ denotes the open immersion $j \colon (X \times X^n \times X) \setminus (\bigcup_{i=0}^n A_i) \hookrightarrow X \times X^n \times X$.
For any \'etale cover $\pi \colon Y \rightarrow X$, we set $\sC(n)(Y)=p^\star \sC(n)$.
\end{defn}

For later use, we note that $R^qj_\star\Lambda(nd)=0$ if~$q$ is not divisible by $2d-1$ and
that $$R^qj_\star\Lambda(nd)=\bigoplus_{J \subseteq \{0,\dots,n\}, \#J=m} \Lambda_{A_J}((n-m)d)$$
if $q=m(2d-1)$ for some integer~$m$.  In particular there are natural distinguished triangles
\begin{equation}
\label{dt1}
\xymatrix@C=3ex{
\tau_{\le (n-1)(2d-1)} Rj_\star \Lambda(nd) \ar[r] & \sC(n) \ar[r] & \displaystyle\bigoplus_{J \subseteq \{0,\dots,n\}, \#J=n} \Lambda_{A_J}[-n(2d-1)] \ar[r]^(.9){+1} &
}
\end{equation}
and
\begin{equation}
\label{dt2}
\xymatrix@C=3ex{
\sC(n) \ar[r] & Rj_\star \Lambda(nd) \ar[r] & \Lambda_{A_{\{0,\dots,n\}}}(-d)[-(n+1)(2d-1)] \ar[r]^(.85){+1} & {\;\;\text{.}}
}
\end{equation}

The following proposition shows that in the algebraic setting,
the $n$\nobreakdash-th hypercohomology group of $\sC(n)$ twisted by~$2d$ plays a r\^ole
analogous to that of the vector space $\Q[\pitop(X,a)]/I^{n+1}$ in Beilinson's original construction.  We shall not pursue this analogy further.
Instead, we shall use the complexes~$\sC(n)$   in~\textsection\ref{sub3.4} to define liftings of the class $\alpha(\Lambda)$ associated in Theorem~\ref{thm2.5} to a section
of $\pi_1(X,a) \rightarrow G_k$.

\begin{prop}\label{prop3.6}
Assume~$X$ is a~\kpi.  For any $a \in X(\kbar)$,
the $\Lambda$\nobreakdash-module
$$\varprojlim \H^{2(n+2)d-n}_c(\Ybar\times  \Xbar^n\times  \Ybar, \sC(n)(\Ybar) \otimes_\Lambda \Lambda(2d))\text{,}$$
where the inverse limit ranges over all factorisations
$\Xtilde_a \rightarrow \Ybar \rightarrow \Xbar$
of~$\pitilde_a$ with~$\Ybar$ finite over~$\Xbar$,
and where the transition maps are the trace maps,
is canonically dual
to
$\H^n(\Xbar^n,\sB^e(n)_{aa})$.
\end{prop}

\begin{proof}
There is a canonical quasi-isomorphism $\sC(n)=R\mathcal{H}\text{om}(\sB^e(n),\Lambda(nd))$.  (More generally, applying $R\mathcal{H}\text{om}(-,\Lambda(nd))$ to the
distinguished triangle
$$
\xymatrix{
\Lambda_{A_{\{0,\dots,n\}}}[-n-1] \ar[r] & j_! \Lambda \ar[r] & \sB^e(n)  \ar[r]^(.7){+1} &
}
$$
yields~(\ref{dt2}).)  Hence, by Poincar\'e duality, a perfect pairing of $\Lambda$\nobreakdash-modules
$$
\H^{2(n+2)d-n}_c(\Ybar\times  \Xbar^n\times  \Ybar, \sC(n)(\Ybar)) \times \H^n(\Ybar \times \Xbar^n \times \Ybar,\sB^e(n)(\Ybar)) \longrightarrow \Lambda(-2d) \text{.}
$$
for any \'etale cover $\Ybar \rightarrow \Xbar$.  We conclude by applying Proposition~\ref{prop3.4}.
\end{proof}

\subsection{Liftings of the \texorpdfstring{classes $\alpha(\Lambda)$}{cycle classes}}\label{sub3.4}
The aim of~\textsection\ref{sub3.4} is to use the definitions of~\textsection\ref{sub3.3} to construct liftings of the classes $\alpha(\Lambda)$ associated in Theorem~\ref{thm2.5} to sections
of $\pi_1(X,x) \rightarrow G_k$.
We retain the notations and hypotheses of~\textsection\ref{sub3.3}.  In addition, we assume throughout that~$X$ is a~\kpi{}
and we fix a section~$s$ of the natural map $\pi_1(X,x) \rightarrow G_k$ for some $x \in X(\kbar)$.
For simplicity we also assume that~$X$ is proper.

First we remark that~(\ref{dt1}) yields, for $n=1$, a map
$$
H^0(A_0,\Lambda) \oplus H^0(A_1,\Lambda) \longrightarrow H^{2d}(X \times X \times X, \Lambda(d))
$$
and hence two classes in $H^{2d}(X\times X \times X,\Lambda(d))$.
Concretely, these two classes are the classes of the algebraic cycles $A_i \subset X \times X \times X$ for $i\in\{0,1\}$.
Let $c_0$, $c_1$ denote their images in
$\varinjlim H^{2d}(Y \times X \times Y,\Lambda(d))$,
where the direct limit ranges over all
factorisations $X_s \rightarrow Y \rightarrow X$ of $\pi_s \colon X_s \rightarrow X$ with~$Y$ finite over$~X$.
As in Proposition~\ref{prop2.2}, we have $\varinjlim H^{2d}(Y \times X \times Y,\Lambda(d)) = H^{2d}(X,\Lambda(d))$.
Now it~follows from the construction of~$\alpha(\Lambda)$ that~$c_0$ and~$c_1$ coincide, via this equality, with $\alpha(\Lambda) \in H^{2d}(X,\Lambda(d))$.
Hence $\alpha(\Lambda)$ can be read off of~(\ref{dt1}) for $n=1$; which begs us to also consider this triangle for~$n>1$.

\begin{defn}
Let $n \geq 1$.
We define~$\sC'(n)$ by the formula
$$
\sC'(n)=\tau_{\le n(2d-1) } Rj'_\star\Lambda(nd)
$$
where~$j'$ denotes the open immersion 
 $j' \colon (X \times X^n \times X) \setminus (\bigcup_{i=0}^{n-1} A_i) \hookrightarrow X \times X^n \times X$.
For any \'etale cover $\pi \colon Y \rightarrow X$, we set $\sC'(n)(Y)=p^\star \sC'(n)$.
\end{defn}

The complexes~$\sC(n)$ and~$\sC'(n)$ are related by a distinguished triangle
\begin{equation}
\label{dt3}
\xymatrix{
\sC'(n) \ar[r] & \sC(n) \ar[r]^(.27){\mathrm{res}} & i_\star \sC(n-1)[-(2d-1)] \ar[r]^(.82){+1} &
}
\end{equation}
where~$i$ denotes the closed immersion $i \colon X \times X^{n-1} \times X \simeq A_n \hookrightarrow X \times X^n \times X$
and $\mathrm{res}$ is the residue map.  (We take the convention that $\sC(0)=\Lambda$.)

\begin{defn}
For any $n \geq 1$, we define
$$c_{1,\ldots,n} \in \varinjlim \H^{n(2d-1)+1}(Y\times X^n\times Y, \tau_{\le (n-1)(2d-1)}p^\star Rj_\star\Lambda(nd))\text{,}$$
where the direct limit ranges over all factorisations $X_s \rightarrow Y \rightarrow X$ of $\pi_s \colon X_s \rightarrow X$ with~$Y$ finite over~$X$,
to be the class of the image of~$1$ by the map $$H^0(A_{\{1,\dots,n\}},\Lambda) \longrightarrow 
\H^{n(2d-1)+1}(X\times X^n\times X, \tau_{\le (n-1)(2d-1)}Rj_\star\Lambda(nd))$$ stemming from~(\ref{dt1}).
\end{defn}

\begin{prop}\label{prop3.11}
The class $c_{1,\dots,n}$ is a lifting of~$\alpha(\Lambda) \in H^{2d}(X,\Lambda(d))$ (by an ``iterated residue'' map).
\end{prop}
\begin{proof}
Let $n \geq 2$.
For any $Y \rightarrow X$ appearing in~$\pi_s$, the residue map in~(\ref{dt3}), together with~(\ref{dt1}), gives rise to a commutative square
$$
\mbox{
\kern-1cm
\xymatrix{
\ar[d] \displaystyle\bigoplus_{J \subseteq \{0,\dots,n\},\#J=n} H^0(A_J, \Lambda) \ar[r] & \ar[d] \H^{n(2d-1)+1}(Y\times X^n\times Y, \tau_{\le(n-1)(2d-1)}p^\star Rj_\star\Lambda(nd))\\ 
\displaystyle\bigoplus_{J \subseteq \{0,\dots,n-1\},{\#J=n-1}} \mkern-32mu H^0(A_J, \Lambda)\ar[r] & \H^{(n-1)(2d-1)+1}(Y\times X^{n-1}\times Y, \tau_{\le(n-2)(2d-1)}p^\star Rj_\star\Lambda((n-1)d))
}
}
$$
where the vertical map on the left sends $1 \in H^0(A_J,\Lambda)$
to $1 \in H^0(A_{J \setminus \{n\}},\Lambda)$ if $n \in J$, to~$0$ otherwise,
and the vertical map on the right is the residue along $Y \times X^{n-1} \times Y = p^{-1}(A_n) \subset Y \times X^n \times Y$.
As a consequence, for every $n \geq 2$, the class $c_{1,\dots,n}$ maps to $c_{1,\dots,n-1}$ by the residue map.  Since $c_1=\alpha(\Lambda)$, this proves the proposition.
\end{proof}

As an example, let us consider the case $n=2$ in more detail.
The class $c_{1,2}$ lives in $L:=\varinjlim\H^{4d-1}(Y\times X^2\times Y, \tau_{\le 2d-1}p^\star Rj_\star\Lambda(2))$.
One may consider its residues along three subvarieties of $\varprojlim (Y \times X^2 \times Y)$, namely $\varprojlim p^{-1}(A_i)$ for $i \in \{0,1,2\}$.
These three residue maps fit into the exact sequence
\begin{equation}
\label{sedern}
\xymatrix{
L \ar[r] & \displaystyle\bigoplus_{i\in\{0,1,2\}} H^{2d}(X,\Lambda(d)) \ar[r] &
H^{4d}(X^2,\Lambda(2d))
}
\end{equation}
obtained by applying $\varinjlim\H^{4d-1}(Y \times X^2\times Y,p^\star-)$ to the distinguished triangle
$$
\xymatrix{
\tau_{\le 2d-1}Rj_\star \Lambda(2d) \ar[r] & \displaystyle\bigoplus_{i\in\{0,1,2\}} \Lambda_{A_i}(d)[-(2d-1)] \ar[r] & \Lambda(2d)[1] \ar[r]^(.63){+1} & {\;\;\text{.}}
}
$$
A simple calculation reveals that up to a sign, the image of~$c_{1,2}$ in the middle group of~(\ref{sedern}) is $(0,-\alpha,\alpha)$
and the second map of~(\ref{sedern}) is given by $(x,y,z) \mapsto \alpha \boxtimes x + \Delta_\star y + z \boxtimes \alpha \in H^{4d}(X^2,\Lambda(2d))$.
Here~$\alpha$ stands
for~$\alpha(\Lambda)$, and $\boxtimes$ and $\Delta_\star$ respectively denote the exterior product and the Gysin map associated to the
diagonal embedding $X \subset X^2$.  Since the image of~$c_{1,2}$ in the right-hand side group of~(\ref{sedern}) vanishes,
we deduce that the cycle class~$\alpha$ satisfies the nontrivial relation
\begin{equation}
\label{boxtimesgysin}
\alpha \boxtimes \alpha = \Delta_\star \alpha
\end{equation}
in $H^{4d}(X^2,\Lambda(2d))$.  Hence we see that the class $c_{1,2}$ contains more information about~$\alpha$ than~$c_1$ alone.
More generally, one might hope to use all of the liftings of~$\alpha$ defined in Proposition~\ref{prop3.11} in order to ``rigidify'' the $0$\nobreakdash-cycle of degree~$1$ constructed
up to linear equivalence in Proposition~\ref{propkoenig} (so as to show that it lies in $X(k)\subset \Pic^1(X)$).

\begin{rmks}
(i) Cycle classes of rational points satisfy~(\ref{boxtimesgysin});
the situation for cycle classes in $H^{2d}(X,\Lambda(d))$ of arbitrary $0$\nobreakdash-cycles of degree~$1$
is less clear.

(ii) Assume~$X$ is a curve.  Denote by $[\Delta] \in H^2(X^2,\Lambda(1))$ the class of the diagonal and by $\omega \in H^2(X,\Lambda(1))$ the class of the canonical
sheaf.  Then $\Delta_\star \alpha = p^\star\alpha \cup [\Delta]$,
where $p \colon X \times X \rightarrow X$ denotes the first projection, and hence $\Delta^\star \Delta_\star \alpha = \alpha \cup (-\omega)$
by adjunction.  Applying~$\Delta^\star$ to~(\ref{boxtimesgysin}) therefore yields the equality $\alpha \cup (\alpha + \omega)=0$,
which we had already encountered in Lemma~\ref{lemmaalphaomega} (see Remark~\ref{rmks37}~(i)).
\end{rmks}

\bibliographystyle{plain}
\renewcommand\refname{References}

\end{document}